\documentclass[reqno]{article}
\usepackage{epsfig}
\usepackage{latexsym}
\usepackage{fullpage}
\usepackage{amsthm,amsmath,amssymb}
\usepackage{graphicx} 
\usepackage{subfigure} 
\usepackage{color}
\usepackage{algorithm}
\usepackage{algorithmic}

\newcommand{\real}{\mathbb{R}}

\newtheorem{theorem}{Theorem}
\newtheorem{definition}{Definiton}
\newtheorem{lemma}{Lemma}

\newcommand{\A}{\mathcal{H}}
\newcommand{\Se}{\mathcal{Y}}
\newcommand{\Sn}{\mathcal{S}^{n-1}}
\newcommand{\F}{\mathcal{F}}

\newcommand{\ones}{{\bf1}}

\newcommand{\prox}{\text{prox}}

\begin{document}

\title{Convergence of a Steepest Descent Algorithm  for\\ Ratio Cut Clustering}

\author{ Xavier Bresson\thanks{Department of Computer Science, City University of Hong Kong, Hong Kong ({\tt xbresson@cityu.edu.hk}).},   Thomas Laurent\thanks{Department of Mathematics, University of California Riverside, Riverside CA 92521  ({\tt laurent@math.ucr.edu})}, David Uminsky\thanks{Department of Mathematics, University of California Los Angeles, Los Angeles CA 90095 ({\tt duminsky@math.ucla.edu})} and James H. von Brecht\thanks{Department of Mathematics, University of California Los Angeles, Los Angeles CA 90095 ({\tt jub@math.ucla.edu})}  }

\maketitle

\abstract{Unsupervised clustering of scattered, noisy and high-dimensional data points is an important and difficult problem. Tight continuous relaxations of balanced cut problems have recently been shown to provide excellent clustering results. In this paper, we present an explicit-implicit gradient flow scheme for the relaxed ratio cut problem, and prove that the algorithm converges to a critical point of the energy. We also show the efficiency of the proposed algorithm on the two moons dataset.  }


\section{Introduction}

Partitioning data points into sensible groups is a fundamental problem in machine learning  and has a wide range of applications. An efficient approach to deal with this problem is to cast the data partitioning problem as a graph clustering problem. Given a set of data points $V=\{ x_1, \ldots, x_n\}$ and similarity weights $\{w_{i,j}\}_{ 1\le i,j\le n}$, the clustering problem aims at finding a balanced cut of the graph of the data. 
 In this work, we consider the balanced cut of Hagen and Kahng \cite{ratiocut} known as ratio cut. The ratio cut problem is

\begin{align}\label{rc1}
& \text{Minimize } \;\;  \text{RatioCut}(S)= \frac{\sum_{{x_i}\in S} \sum_{{x_j}\in S^c} w_{i,j}}{  |S|}+\frac{\sum_{{x_i}\in S} \sum_{{x_j}\in S^c} w_{i,j}}{  |S^c|} \\
& \text{over all subsets $S\subsetneq V$.} \nonumber
\end{align}
 Here $|S|$ denotes the number of data points in $S$. While the  problem, as stated above,  is NP-hard, it has the following \emph{tight continuous} relaxation: 
\begin{align} \label{rc2}
&\text{Minimize }   \;\;  E(f)= \frac{ \frac{1}{2}\sum_{i,j}  w_{i,j} |f_i-f_j|  }{ \sum_i |f_i-m(f)| } \\
& \text{over all non-constant functions $f: V \to \real$.} \nonumber
\end{align}
Here $m(f)$ stands for the average of $f \in \mathbb{R}^n$ and $f_i$ stands for $f(x_i)$. Recently, various algorithms have been proposed
\cite{pro:SzlamBresson10,pro:HeinBuhler10OneSpec,pro:HeinSetzer11TightCheeger,art:BressonTaiChanSzlam12TransLearn, pro:Rang-Hein-constrained}
to minimize relaxations of balance cut problem similar to \eqref{rc2}. In this work, we present an explicit-implicit gradient flow algorithm, then prove that the iterates converge to critical points of the energy. We also present numerical experiments to show the robustness and efficiency of the algorithm.

\subsection{The Tight Continuous Relaxation}

We begin by first explaining  the meaning of the term \emph{tight relaxation}. Since $E$ is invariant under the addition of a constant, problem \eqref{rc2} is equivalent to

\begin{align} \label{rc3}
&\text{Minimize }  \;\; \;\;   \frac{ \frac{1}{2}\sum_{i,j}  w_{i,j} |f_i-f_j|  }{ \sum_i |f_i| }
  \\  
&\text{ over all $f: V \to \real$ s.t. $m(f) =0$ and $f \neq 0$.} \nonumber
\end{align}
If the graph is connected then the total variation functional  $\frac{1}{2}\sum_{i,j}  w_{i,j} |f_i-f_j|$ defines a norm on the space of mean zero functions; we denote it by $\|f\|_{TV}$.  The denominator of \eqref{rc3} is simply the $\ell^1$-norm, and we denote it by  $\|f \|_1$.

The continuous problem \eqref{rc3}  is a tight relaxation of \eqref{rc1} in the following sense--- if $S^*$ is a solution of $\eqref{rc1}$, then any nonzero, binary function of mean zero
\begin{equation}
f^*(x_i)=\begin{cases} a & \text{ if } x_i \in S^* \\b  & \text{ if } x_i \in (S^*)^c 
\end{cases}
\end{equation}
is a solution of problem \eqref{rc3}. This is a consequence of the fact that the the extreme points of the TV-unit ball
$$
\{f \in \real^n:  \|f\|_{TV} \le 1, m(f)=0 \}
$$
are binary functions (see \cite{pro:SzlamBresson10} for a proof of this fact). Therefore, if we fix $\|f\|_{TV} = 1$ and maximize the convex functional in the denominator of \eqref{rc3}, the minimum of the ratio is attained at an extreme point. That is, at a binary function of mean zero. Binary functions of mean zero are always of the form
$$
f= \lambda \left( |S^c| \chi_{S}-  |S|\chi_{S^c} \right), \; \; S \subsetneq V, \;\; \lambda \neq 0,
$$
where $\chi_S$ is the characteristic function of the set $S$. For such a function, we easily check that $ E(f)=\text{RatioCut}(S)/2$.
From this observation we can see that  if  $S^*$ is a solution of the ratio cut problem \eqref{rc1}, then
$f^*= \lambda \left( |(S^*)^c| \chi_{S^*}+ |S^*|\chi_{(S^*)^c} \right)$ is a solution of the continuous relaxation \eqref{rc3} for any $\lambda \neq 0$.
 A different proof of the fact that problem  \eqref{rc2} is a tight relaxation of problem  \eqref{rc1} can be found in \cite{pro:Rang-Hein-constrained}.

\subsection{Explicit-implicit gradient Flow}
Let 
\begin{equation}
T(f)=  \|f\|_{TV} \text{ and } B(f)= \sum_i |f_i-m(f)|.
\end{equation}
Note that both $T$ and $B$ are convex.
If $T$ and $B$ were differentiable,   the explicit-implicit gradient flow of $E=T/B$ would be
\begin{equation} \label{continuous-grad0}
\frac{f^{k+1}-f^{k}}{\tau^k}= -  \frac{ \nabla T (f^{k+1})- E(f^k) \nabla B(f^k) }{B(f^k)}
\end{equation}
 where $\tau^k$ is the time step. Since $T$ and $B$ are not differentiable, 
we replace \eqref{continuous-grad0} with its non-smooth equivalent:
  \begin{align}
&g^k =f^{k}+ \frac{\tau^k}{B(f^k)} E(f^k) v^{k}  \quad \text{for some } v^{k} \in \partial B (f^{k}) 
\label{algo11}\\
&f^{k+1}= \arg \min_f \left\{ T(f)+ B(f^k) \frac{\|f-g^k\|^2}{2 \tau^k} \right\} \label{algo22}.
\end{align}
 The minimization problem \eqref{algo22} is a standard  ROF problem \cite{art:RudinOsherFatemi92ROF} that can be solved efficiently  using  approaches such as augmented Lagrangian method \cite{art:GoldsteinOsher09SB} or primal-dual method \cite{art:ChambollePock11FastPD}.
The scheme  \eqref{algo11}--\eqref{algo22}, as will be shown in the next section, decreases the energy  and preserve the zero mean properties of the successive iterates.
In order to remain away from the origin, where the energy is not defined, we project each iterate onto the sphere $\Sn=\{u \in \real^n: \|u\|_2=1\}$ at the end of each step. In numerical experiments we observe faster convergence when the time step is chosen to be
\begin{equation}
\tau^k=c\frac{B(f^k)}{E(f^k)}, \quad c>0.
\end{equation}

With these choices, we arrive at our proposed algorithm to find critical points of the ratio cut functional \eqref{rc2}:
  \begin{align}
&g^k =f^{k}+ c v^{k}  \quad \text{for some } v^{k} \in \partial B (f^{k}) \label{algo1}\\
&h^{k}= \arg \min_f \left\{ T(f)+ E(f^k) \frac{\|f-g^k\|^2}{2 c} \right\} \label{algo2} \\
&f^{k+1}= \frac{h^k}{\|h^k\|_2},
\end{align}
which we formalize in Algorithm 1.

\begin{algorithm}[!] 
\caption{Steepest descent of the RatioCut functional \eqref{rc2}}
\label{alg-ratiocut}
\begin{algorithmic}
\STATE $f^{k=0}$ nonzero function with mean zero. \\
$c$ positive constant. 
\WHILE{ loop not converged}
\STATE $w^k \in \text{sign} (f^k),   \quad  v^k = w^k- m(w^k) , \quad \lambda^k= \frac{\|f^k\|_{TV}}{\|f^k\|_1}$ \\
\STATE{{$g^{k}=f^{k}+c\, v^k$} 
}
\STATE{$h^{k}=\arg\min_f \ \{ \|f\|_{TV}+\frac{\lambda^k}{2c}||f-g^{k}||_2^2 \ \}$
}
 \STATE{$f^{k+1}= \frac{h^k}{\|h^k\|_2} $}
\ENDWHILE
\end{algorithmic}
\end{algorithm}

Let $\{f^k\}$ denote a sequence of iterates generated by Algorithm 1, starting from a non-zero function $f^0$ with $m(f^0) = 0$. In section 2, we show that any accumulation point of this sequence is a critical point of the the ratio cut functional \eqref{rc2}. Moreover we show that $\|f^{k+1}-f^{k}\|_2 \to 0$ as $k \to\infty$, so that either the sequence converges or the set of accumulation points is a connected subset of the sphere $\Sn$. In section 3 we demonstrate the efficiency of Algorithm 1 on the two moons example.



\section{Convergence}

Given a connected graph, we want to minimize
$$
E(f)= \frac{\sum_{i,j=1}^n w_{i,j} |f_i-f_j|}{ \sum_{i=1}^n | f_i-m(f)|}= \frac{T(f)}{B(f)}
$$
over the space of non-constant functions $f \in \real^n$. (Note that $E$ is not defined for constant functions).
This is equivalent to  minimizing $E$ over the set of non-constant functions with mean zero, which we write as
$$
\F=\{f \in \real^n: m(f)=0 \text{ and } f \neq 0\}.
$$
We define $\ones := (1,\ldots,1)^{T} \in \real^n,$ so that $m(f) = \langle\ones,f\rangle/n$ and $\ones^\perp$ gives the space of functions with mean zero. Clearly $\F$ is an open subset of $\ones^\perp$. As we assume a connected graph, $T$ and $B$ define norms on $\ones^\perp$. Since all norms are equivalent in finite dimensions, there exist constants $\beta>  \alpha>0$ such that
$$
\alpha B(f) \le T(f) \le  \beta B(f) \quad \text{ for all } f \in \ones^\perp.
$$
Therefore
$$
\alpha \le E(f) \le \beta \quad \text{ for all } f \in \F.
$$
If we let
\begin{align}
 L(f)= \|f\|_1=\sum_{i=1}^n |f_i|, \;\; P_{0} f= f-m(f) {\bf1},
 \end{align}
then we see that $B(f)=L(P_{0}f)$. Note that
$
P_{0} =\text{Id}-\frac{1}{n}\ones \ones^{T},
$
so that the matrix $P_{0}$ simply gives the orthogonal projection onto $\ones^\perp$. As $L(f)$ is convex, so is $B(f)=L(P_{0 } f)$, and we also have
$$
\partial B(f)= P_{0 } \;  \text{sign}(P_{0}f).
$$
It is then easy to see that $\langle \partial B(f), {\bf 1} \rangle =0 $ for all $f$.  If $f\in \ones^\perp$, then $B(f)= L(f)$ and $\partial B(f)$ is simply the projection of $\partial L(f)$ on $\ones^\perp$, i.e. $\partial B(f)= P_0 \; \text{sign}(f)$.

Starting from a non-constant function $f$, we define $g$ and $h$ according to Algorithm \ref{alg-ratiocut}
\begin{align}
&{g} =f+ c v, \quad \text{ where } \quad  v \in \partial B (f) \label{algo1} \\
&h= \arg \min_u \left\{ T(u)+ E(f) \frac{ \|u-g\|_2^2}{2  c} \right\} \label{algo2}, \quad
\end{align}
which we write succinctly as
$$
h\in \A^c(f).
$$
Since $g$ is not uniquely defined when $B(f)$ is non-differentiable, in general $\A^c(f)$ may have more than one element. Therefore the map $\A^c$ is a {\it set-valued map} defined over the space of non-constant functions (see Definition \ref{SVM} in the following subsection).

\subsection{Estimates}

\begin{lemma}[Elementary properties of $\A^c$] \label{properties}  Let  $g$ and $h$ be defined  by \eqref{algo1}--\eqref{algo2}.
\begin{enumerate}
\item If $f$ is not constant, then  $h$ is not constant. Moreover, the energy inequality
\begin{equation} \label{descent}
E(f)  \ge E(h) + \frac{E(f)}{B(h)}\frac{\|h-f\|_2^2}{c} 
\end{equation}
holds. As a consequence, $E(h) < E(f)$ unless $h = f$.
\item If $f$ is not constant, then
\begin{equation} \label{upperbound}
\|h\|_2 \le \|g\|_2 \le \|f\|_2+  2c \sqrt{n}.
\end{equation}
\item If $f \in \real^n$, then  $\|g\|_2 > \|f\|_2$, or, to be more precise:
$$\|g\|_2^2 =
 ||f||_2^2 + 2c B(f) + c^{2} || \partial B(f) ||_2^{2}.
$$
\item If $f \in \F$, then $g,h \in \F$.
\end{enumerate}
\end{lemma}
\begin{proof} (1.) The definition \eqref{algo2} of $h$ implies that
$
E(f) \frac{h-g}{ c} \in - \partial T(h),
$ and therefore,
since $T$ is convex, 
\begin{align}
T(f) &\ge T(h)+ \left\langle -E(f)\frac{h- g }{c}, f-h \right\rangle \quad \\
&= T(h)-E(f)\left\langle \frac{h- ({f}+ c  v) }{c }, f-h \right\rangle \\
 &= T(h)+  \frac{E(f)}{ c }\|h-f\|^2_2 -  E(f)
\left\langle  v , h-f \right\rangle.  \label{T}
\end{align}
Since $B$ is also convex, we have
$
B(h) \ge B(f)+ \langle v, h-f \rangle  \label{B},
$
and therefore adding these two last inequalities,
$$
T(f)+ E(f) B(h) \ge T(h)+ E(f) B(f) + \frac{E(f)}{c }\|h-f\|^2_2.$$
In other words, 
$$
 E(f) B(h) \ge T(h) + \frac{E(f)}{ c }\|h-f\|^2_2.
$$
Since $f$ is not constant, we have  $E(f)>0$.  Note that if $h$ were constant, then $B(h)=0$ which would imply $h=f$. This is a contradiction since $f$ is not constant. Thus $B(h)>0,$ so we may divide in the last expression to obtain \eqref{descent}.

\

\noindent  (2.)  To prove that $\|h\|_2 \le \|g\|_2$, note 
$$
h=\prox_{\Phi}(g) : = \arg \min_u \left\{ \Phi(u)+  \frac{ \|u-g\|_2^2}{2} \right\} \qquad \text{ where } \Phi(u)= \frac{c}{E(f)} T(u).
$$
Since proximal mappings are Lipshitz continuous with constant one, and since $\prox_{\Phi}(0)=0$,  we have
\begin{equation} \label{contract}
\|h\|_2=\|\prox_\Phi(g)-\prox_\Phi(0)\|_2 \le \|g\|_2.
\end{equation}
To establish the inequality $\|g\|_2 \le \|f\|_2+  2c \sqrt{n}$, note that $\|\text{sign}(P_0 f)\|_\infty \le 1$ and therefore 
\begin{equation} \label{boundsubgrad}
\|\partial B(f)\|_2 \le \sqrt{n}\|\partial B(f)\|_\infty= \sqrt{n} \|\text{sign}(P_0 f) - m(\text{sign}(P_0 f)) \ones\|_\infty \le 2 \sqrt{n}  \quad \text{ for all } f \in \real^n.
\end{equation}
The upper bound then follows from the definition of $g$ and the triangle inequality. 

\

\noindent (3.) Since $B$ is homogeneous of degree one, we have
\begin{multline}
	||g||_2^2 
	= ||f+ c \partial B (f)||_2^{2} 
	=  ||f||_2^2 + 2 c \left<f , \partial B (f) \right> +c^{2} || \partial B(f) ||_2^{2} 
	=  ||f||_2^2 + 2c B(f) + c^{2} || \partial B(f) ||_2^{2}.
\label{eq:g-dist}
\end{multline}	

\

\noindent (4.)  Since $\partial B(f) \subset \ones^\perp,$ it is clear  that $f \in \ones^\perp$ implies $g \in \ones^\perp$. Equation \eqref{eq:g-dist} shows that $||g||_2>||f||_2>0$ so that $g$ cannot be constant (the only constant function of mean zero is the zero function). Thus $g \in \F$. Suppose that $h \notin \ones^\perp$. Since $P_0$ projects onto $\ones^\perp$ and since $T(P_0u)=T(u)$ for all $u \in \real^n$ (because $T$ is invariant under addition of a constant), we have 
$$T(h)+\frac{ E(f) }{2 c} \|h-g\|_2^2=T(P_0 h) + \frac{ E(f) }{2 c} \left( \|P_0 h-g\|_2^2+ \|(\mathrm{Id}-P_0 )h\|_2^2 \right).$$
This contradicts the definition of $h$ as the global minimizer unless $(\mathrm{Id}-P_0 )h = 0.$ Thus $h$ has mean zero. By property (1.) we know $h$ is not constant, so $h \in \F$ as well.
\end{proof}

\begin{definition}
Let $f^0 \in \F$. We say that $f^k, g^k,h^k$ is a sequence generated by the algorithm if
$$
f^{k+1} \in  P_{2} (\A^{c}(f^k)) \quad  \text{ where $P_2$ is the projection onto the sphere $\Sn$} 
$$
and where $g^k$ and $h^k$ are defined from $f^k$ by \eqref{algo1} and \eqref{algo2}.
\end{definition}
\begin{lemma}[Properties of the iterates] If $f^k, g^k,h^k$ is a sequence generated by the algorithm, then $E(f^{k+1}) \le E(f^k)$ with equality if and only if $f^{k}=f^{k+1}$. Moreover,
\begin{equation} \label{attractor}
 \|f^k-h^k\|_2 \to 0 \quad \text{and} \quad   \|f^k-f^{k+1}\|_2 \to 0.
 \end{equation}
Therefore $\Sn$ is an attractor for the sequence $\{h^k\}$.
\end{lemma}
\begin{proof} The fact that the energy decreases is a consequence of \eqref{descent} from Lemma \ref{properties} together with the fact that $E(f^{k+1})=E(h^k)$ due to the invariance of $E$ under scaling. As $f^k \in \ones^{\perp}$ and $||f^k||_2 = 1$ it follows that $E(f^k) \geq \alpha > 0$. From \eqref{descent} we then have
\begin{equation} \label{descent2}
 {\|h^k-f^k\|_2^2}  \le\frac{c}{\alpha}  B(h^k) (E(f^k)-E(f^{k+1})).
\end{equation} 
Now from \eqref{upperbound} we have 
$$B(h^k)=\|h^k\|_1 \le \sqrt{n} \|h^k\|_2 \le  \sqrt{n}+ 2nc,$$ and therefore
$$
{\|h^k-f^k\|_2^2}  \le \frac{c}{\alpha} (\sqrt{n}+ 2nc) (E(f^k)-E(f^{k+1})) \to 0 ,$$
where we have used that $E(f^k)$ is a converging sequence since it is decreasing and bounded from below.

We now show  $\|f^k-f^{k+1}\|_2 \to 0$. Note that the projection $P_2$ is smooth on the annulus  $\mathcal{A} := \{u \in \real^n: 1/2 \le \|u\| \le 3/2\}$ and therefore it is Lipschitz continuous on $\mathcal{A}$ with constant, say, $C$. Since eventually $h^k \in \mathcal{A},$ we have
$$
\|f^k-f^{k+1}\|_2= \| P_2(f^k)-P_2(h^k) \|_2 \le C \|f^k-h^k\|_2 \to 0.
$$  
\end{proof}

\subsection{Proof of convergence}
\begin{definition}[Set-valued map] \label{SVM} Let $X$ and  $Y$ be two subsets of $\real^n$. If for each $x\in X$ there is a corresponding set $F(x)\subset Y$  then $F$ is called a set-valued map from $X$ to $Y$ . We denote this by
$F: X \rightrightarrows Y$. The graph of $F$, denoted $\text{Graph(F)}$ is defined by
$$
\text{Graph}(F)= \{ (x,y) \in \real^n \times \real^n: y \in F(x), x \in X\}.
$$ 
A set-valued map $F$ is called closed if $\text{Graph}(F)$ is a closed subset of $\real^n \times \real^n$.
\end{definition}

Define the compact sets
\begin{align}
&K_1= \{ u \in \real^n: \|u\|_2=1 \text{ and } m(u)=0\} \\
&K_2 = \{ u \in \real^n: 1 \le \|u\|_2\le 1+2 c\sqrt{n} \text{ and } m(u)=0 \} 
\end{align}
along with the set-valued map $\Se^c: K_1 \rightrightarrows  K_2 $
$$
\Se^c(f)=f+c  \partial B (f).
$$
The fact that the range of $\Se^c$ is in $K_2$ is a consequence of \eqref{upperbound}.
\begin{lemma}
The set-valued map $\Se^c$ is closed.
\end{lemma}
\begin{proof} 
Let us first show that the set-valued map $\text{sign}: \real^n \rightrightarrows [-1,1]^n$ is closed.
Let assume that
\begin{align}
& f^k \to f^* \\
& z^k \in \text{sign}(f^k) \to z^* \label{gagaoo}
\end{align}
We want to show that $z^*\in \text{sign}(f^*)$, or equivalently, $z^*_i\in \text{sign}(f^*_i)$ for all $1\le i \le n$. If $f_i^*>0$ then $f^k_i>0$ for $k$ large enough. As $z^k_i=1$ for all such $k$ it follows that $z^*_i=1=\text{sign}(f^*_i)$. Similar reasoning applies if $f_i^*<0$. Lastly, if $f_i^*=0$ then $\text{sign}(f_i^*)=[-1,1]$. The entire sequence $\{z^k_i\}_{k=1}^\infty$ therefore lies in  $\text{sign}(f_i^*),$ so obviously $z^*_i \in \text{sign}({f^*_i})$ as well. 

To show that $\Se^c$ is closed, assume first that
\begin{align}
& f^k \to f^* \\
& g^k \in \Se^c(f^k)=f^k+c\; P_0 \;  \text{sign}(f^k) \to g^* \label{gaga},
\end{align}
where we have used the fact that $\partial B (f)= P_{0}  \; \text{sign}(f)$ whenever $f \in K_{1}$.
Thus our goal is to prove that $g^* \in \Se^c(f^*)$. Clearly there exists $z^k \in \text{sign}(f^k)$ such that
\begin{equation}\label{jospin}
 g^k =f^k+cP_0  z^k.
\end{equation}
Since $z^k$ lies in a compact set  there exists a subsequence $z^{k_i} \to z^*$. So we have
\begin{align}
& f^{k_i} \to f^* \\
& z^{k_i} \in  \text{sign}(f^{k_i}) \to z^*
\end{align}
Since $\text{sign}$ is closed $z^* \in \text{sign}(f^*)$, which combines  with \eqref{jospin} gives
$$
g^{k_i} \to f^*+ c P_0 z^* \in \Se^{c}(f^*)
$$
where we have used the definition of $\Se^c(f^*)$ and the fact that $f^* \in K_1$.  From \eqref{gaga} we then obtain $g^* \in  \Se^{c}(f^*)$ as desired.
\end{proof}
We define the function $\Psi^c: K_1 \times K_2 \to \real^d  $
$$
\Psi^c(f,g) = \arg \min_u \left\{ T(u)+ E(f) \frac{ \|u-g\|_2^2}{2  c} \right\} 
$$

\begin{lemma}
The function $\Psi^c$ is continuous on $K_1\times K_2$.
\end{lemma}
\begin{proof}
Let $h=\Psi^c(f,g)$ and $h'=\Psi^c(f',g')$. Then we have
$
E(f) \frac{h-g}{c} \in - \partial T(h)$ and   $E(f') \frac{h'-g'}{c} \in - \partial T(h')
$
so 
\begin{align*}
&T(h') \ge T(h)- \left\langle E(f) \frac{h-g}{c}, h'-h \right\rangle  \\
&T(h) \ge T(h')- \left\langle E(f') \frac{h'-g'}{c}, h-h' \right\rangle.  
\end{align*}
By adding these two inequalities,
\begin{equation*}
\left\langle E(f) (h-g)-  E(f') {(h'-g'}), h-h' \right\rangle  \le 0.
\end{equation*}
Adding and subtracting  we get 
\begin{align*}
&\Big\langle E(f) (h-g)-  E(f) {(h'-g'}), h-h' \Big\rangle 
+\Big\langle (E(f)-E(f')) (h'-g'), h-h' \Big\rangle
\le 0 \\
&E(f)\Big\langle  (h-h')-  {(g-g'}), h-h' \Big\rangle 
+(E(f)-E(f'))\Big\langle h'-g', h-h' \Big\rangle
\le 0 \\
&E(f) \left( \|h-h'\|^2_2-\Big\langle  g-g', h-h' \Big\rangle \right)
+(E(f)-E(f'))\Big\langle h'-g', h-h' \Big\rangle
\le 0 \\
& \|h-h'\|^2_2 \le \Big\langle  g-g', h-h' \Big\rangle 
- \frac{(E(f)-E(f'))}{E(f)}\Big\langle h'-g', h-h' \Big\rangle
\end{align*} 
From Cauchy-Schwarz we have 
$$
\|h'-h\|_2 \le \|g'-g\|_2+ \frac{|E(f')-E(f)|}{E(f)} \;  \|h'-g'\|_2 \le \|g'-g\|_2+ \frac{|E(f')-E(f)|}{E(f)}2 \|g'\|_2
$$
The last inequality follows from \eqref{contract}. 
We then easily conclude that if $(f',g') \to (f,g)$ then $h' \to h$, due to the continuity of $E$ on $K_1$.
\end{proof}

We next show that the set-valued map $\A^c: K_1 \rightrightarrows  \F $
$$
\A^c(f)=\Psi^c(f, \Se^c(f))
$$
is closed. The fact that the range of $\A$ is in $\F$ is a consequence of Lemma \ref{properties}.
\begin{lemma}
The set-valued map $\A^c$ is closed.
\end{lemma}
\begin{proof}
Suppose that
\begin{align}
&f^k \to f^* \label{pipi}\\
&h^k \in \A^c(f^k)=\Psi^c(f^k, \Se^c(f^k)) \to h^*  \label{zizi}.
\end{align}
We must show that  $h^* \in \A^c(f^*)$.  Clearly there exist $g^k \in \Se^c(f^k)$ such that
$$
h^k = \Psi^c(f^k, g^k).
$$
Since the sequence $g^k$ is in the compact set $K_2$ there exists $g^* \in K_2$ and a subsequence 
$g^{k_i} \to g^*.$
So we have
\begin{align}
&f^{k_i} \to f^* \label{pipi2}\\
&g^{k_i} \in \Se^c(f^{k_i}) \to g^*  \label{zizi2},
\end{align}
from which we conclude that $g^* \in \Se^c(f^*)$ because $\Se^c$ is closed. Now since
$\Psi^c$ is continuous we have 
$$
h^{k_i} = \Psi^c(f^{k_i}, g^{k_i}) \to \Psi^c(f^*,g^*) \in \Psi^c(f^*, \Se^c(f^*)) = \A^c(f^*).
$$
But $h^{k_i} \to h^*,$ so we may conclude $h^* \in \A^c(f^*)$ as desired.
\end{proof}

\begin{definition}[Critical points]\label{def:critical_point}
Let $f \in \F.$ We say that $f$ is a critical point of the energy $E(f)$ if there exist $w \in \partial T(f)$ and $v \in \partial B(f)$ so that
$$
0 = w - E(f) v.
$$
If both $T$ and $B$ are differentiable at $f$ then the subdifferentials $\partial T(f), \partial B (f)$ are single-valued, so we recover the usual quotient-rule
$$
0 = \nabla T (f) - E(f) \nabla B(f).
$$
\end{definition}

\begin{theorem}[Convergence of the algorithm] \label{conv-thm}
Take $f^0 \in \F$ and fix a constant $c>0$. Let $\{f^k\}_{k=0}^{+\infty}  \subset \F$ be a sequence generated by the algorithm. Then
\begin{enumerate}
\item Any accumulation point $f^*$ of the sequence is a critical point of the energy.
\item Either the sequence converges, or the set of accumulation points is a connected subset of $\Sn$.
\end{enumerate}

\end{theorem}
\begin{proof}
(1.)  The proof is inspired by \cite{art:meyer76}. Let $f^{k_i}$ denote a subsequence converging to $f^*$. Since the sequence $\{f^{k_i+1}\}_{i=1}^{\infty}$ lies in a compact set we can extract a further subsequence (still denoted $\{f^{k_{i}+1}\}$) that converges to some function $f'$.  So we have, as $i\to \infty$
\begin{align}
f^{k_{i}} &\to f^* \\
f^{k_{i}+1} & \to f'.
\end{align}
But, because of \eqref{attractor} it must be that $f^*=f'$.  Thus we have
\begin{align}
f^{k_{i}} &\to f^* \\
f^{k_{i}+1} \in P_2(\A^c(f^{k_{i}})) & \to f^* \label{conv}.
\end{align}
Clearly there exist  $h^{k_{i}} \in \A^c(f^{k_{i}})$ such that $f^{k_{i}+1} = P_2(h^{k_{i}})$. Since the $h^{k_{i}}$ eventually lie in the annulus $\mathcal{A} := \{1/2 \le \|u\|_2 \le 3/2\}$, we can assume (upon extracting another subsequence) that the $h^{k_{i} } \in \mathcal{A} \to h^*\in \mathcal{A}$. Therefore we have
\begin{align}
f^{k_{i}}  &\to f^* \\
 h^{k_{i}} \in \A^c(f^{k_{i}}) & \to h^*
\end{align}
and since $\A^c$ is closed $h^* \in \A^c(f^*)$.
Since $P_{2}$ is continuous in the annulus $\mathcal{A}$ and all limit points of $\{h^{k}\}$ lie on $\Sn$, we conclude that
$$
f^{k_i+1}=P_2(h^{k_i}) \to P_2(h^*) = h^* \in \A^c(f^*).
$$
From \eqref{conv} we therefore have  $f^* \in  \A^c(f^*)$.
By definition of $\A^c(f^*)$, if $f^* \in  \A^c(f^*)$ then there exists $y^* \in \Se^c(f^*)$ so that
$$
f^* = \arg \min_u \left\{ T(u)+ E(f^*) \frac{ \|u-y^*\|_2^2}{2  c} \right\}.
$$ 
Therefore there exists $w^* \in \partial T(f^*)$ so that $0 = c w^* + E(f^*)(f^* - y^*)$. By definition of $\Se^c(f^*)$ there exists $v^* \in \partial B(f^*)$ so that
$$
0 = c w^* + E(f^*)(f^* - (f^* + c  v^*) ) = c(w^{*} - E(f^*)v^*).
$$ 
Thus $f^*$ is a critical point of the energy according to definition \ref{def:critical_point}.

\

(2.) For any sequence generated by the algorithm, $||f^{k+1} - f^{k}||_2 \rightarrow 0$ according to lemma \ref{attractor}. Moreover, they lie in the bounded set $\Sn \subset \mathbb{R}^n$. The hypotheses of Theorem 26.1 of \cite{book:ostrowski} are therefore satisfied, giving the desired conclusion.
\end{proof}



\section{Experiments}

We construct the two moons dataset as in \cite{pro:BuhlerHein09pLapla} (Figure \ref{twomoons}). The first moon is a half  circle of radius one in $\real^2$, centered at the origin, sampled with a thousand points; the second moon is an upside down half circle also sampled at a thousand points, but centered at $(1,-1/2)$. The dataset is embedded in $\real^{100}$ by adding Gaussian noise with $\sigma=0.015$. In all experiments we use a $10$ nearest neighbors graph with the self-tuning weights as in \cite{pro:ZelnikPerona04SelfTuning} (the neighbor parameter in the self-tuning is set to $7$ and the universal scaling to $1$). The constant $c$ in Algorithm \ref{alg-ratiocut} is taken to be $c=1/4$.

Clustering results with different initial conditions are shown in Figure \ref{experiments}. Since the energy is not convex there is no guarantee that the algorithm will converge toward the global minimizer of the ratio cut functional. However, for most initial data, the algorithm indeed finds the correct solution in a very small number of iterative steps.

\begin{figure}[!] 
\centering
\subfigure[Two moons dataset]{\includegraphics[width=4cm]{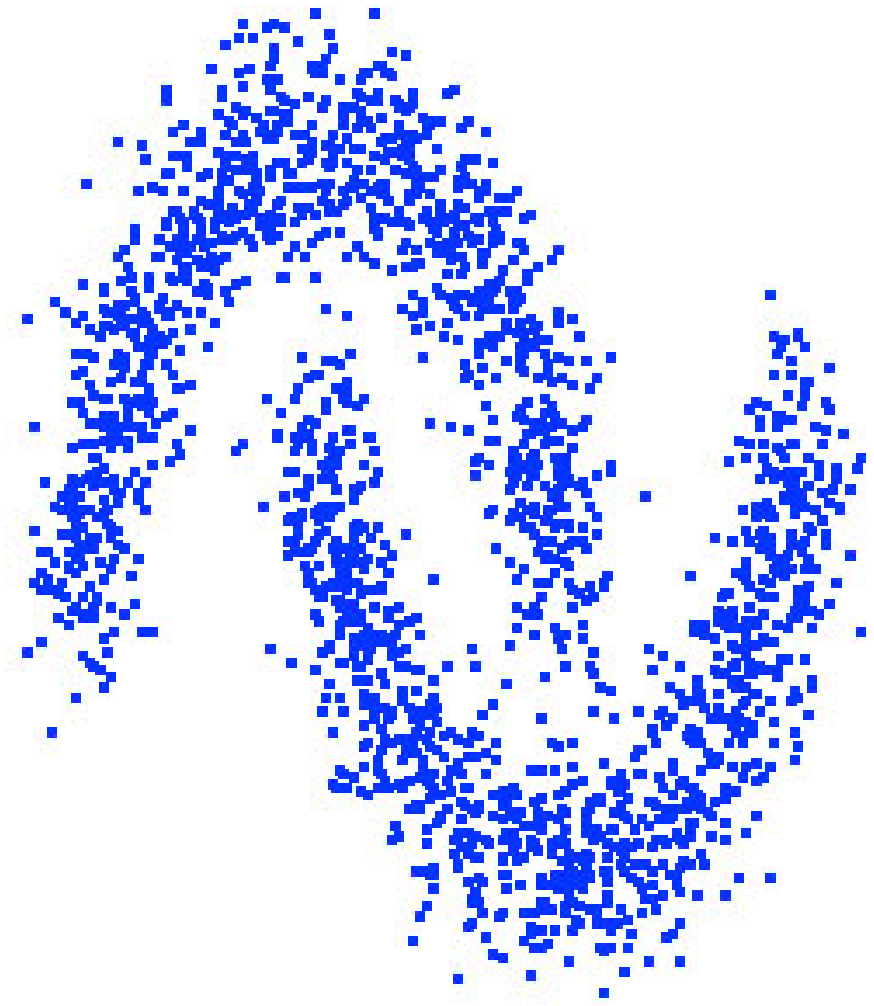}} 
\hspace{0.5cm}
\subfigure[Desired clustering]{\includegraphics[width=4cm]{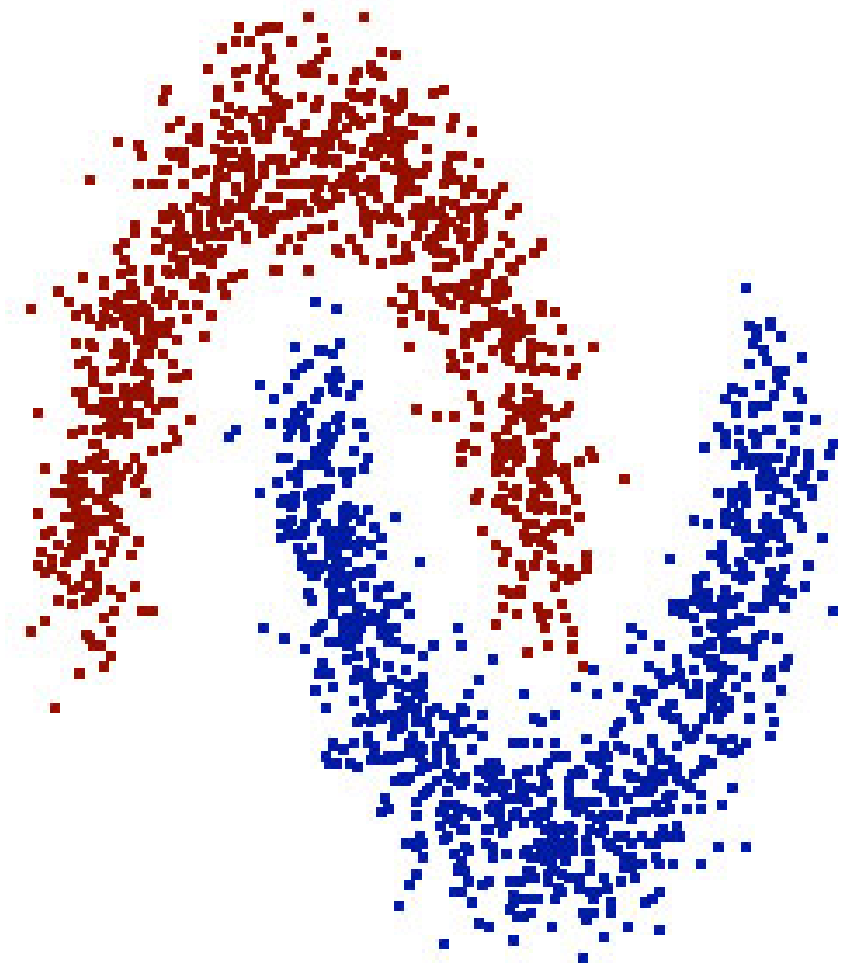}}
\caption{Unsupervised clustering of the two moons dataset. Each moon has 1,000 data points in $\real^{100}$.}
\label{twomoons}
\end{figure}

\begin{figure}[!] 
\centering
\subfigure[Initialization \#1 (2nd eigenvector of graph Laplacian)]{\includegraphics[width=4cm]{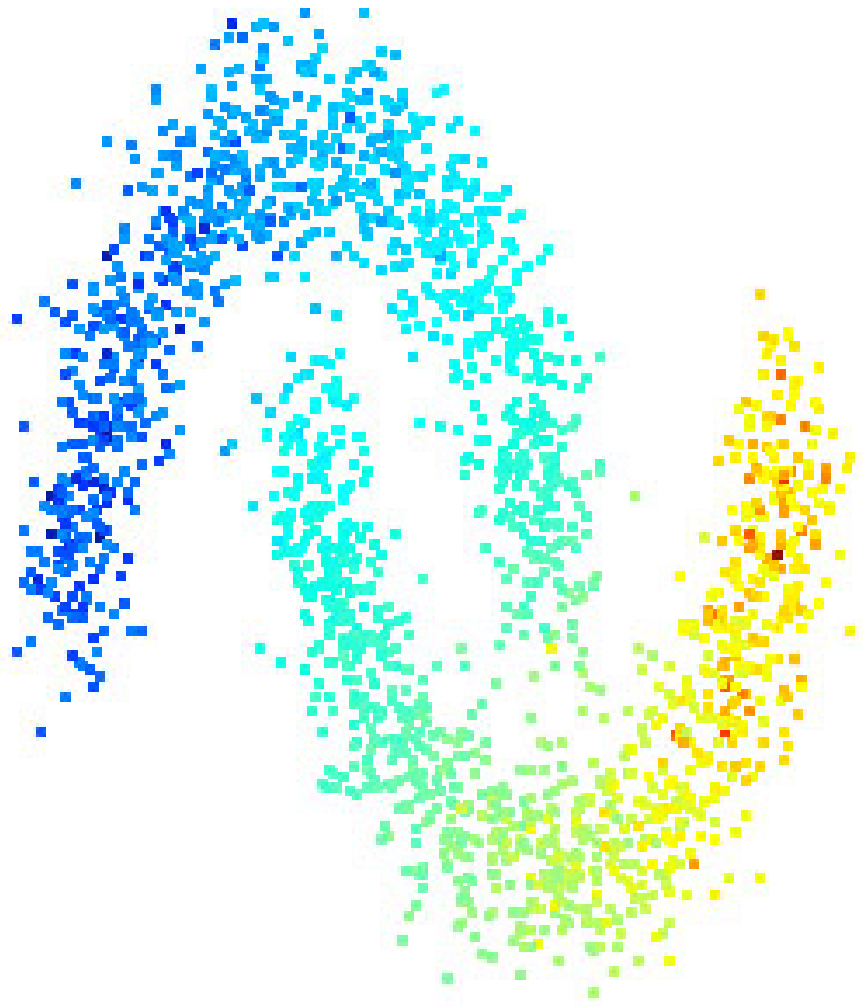}} 
\hspace{0.5cm}
\subfigure[Outcome of Algorithm \ref{alg-ratiocut}]{\includegraphics[width=4cm]{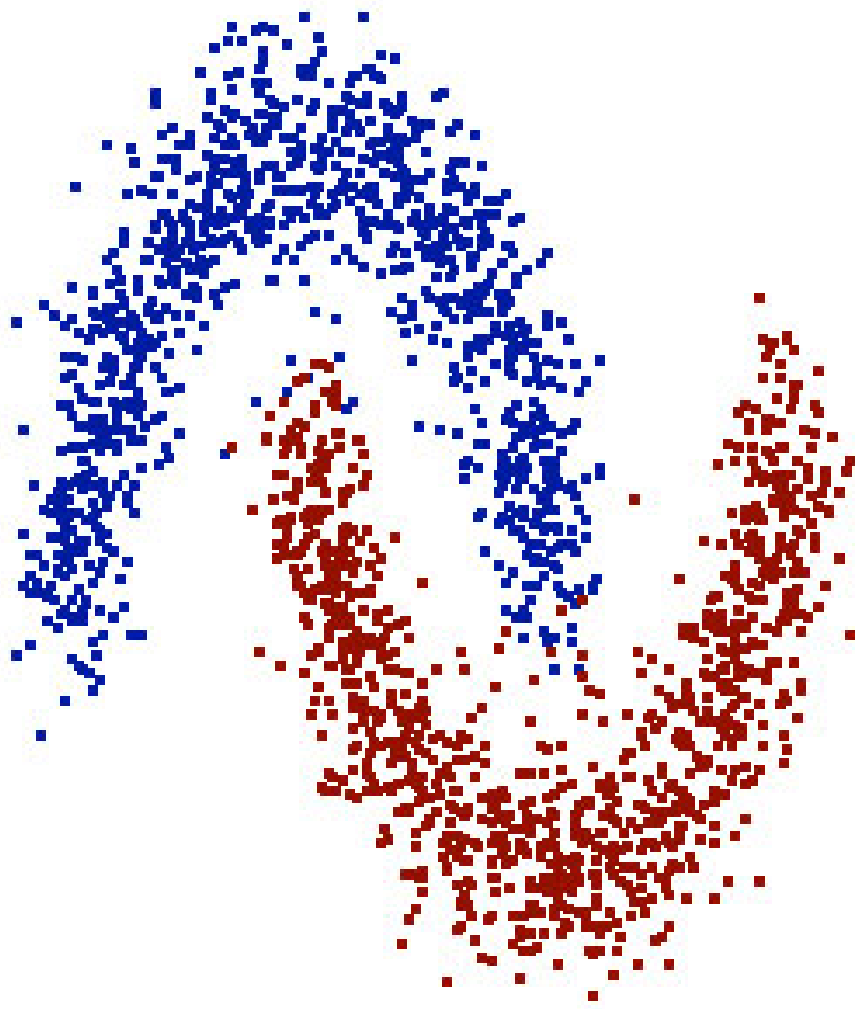}} 
\hspace{0.5cm}
\subfigure[Energy w.r.t. iteration]{\includegraphics[width=4cm]{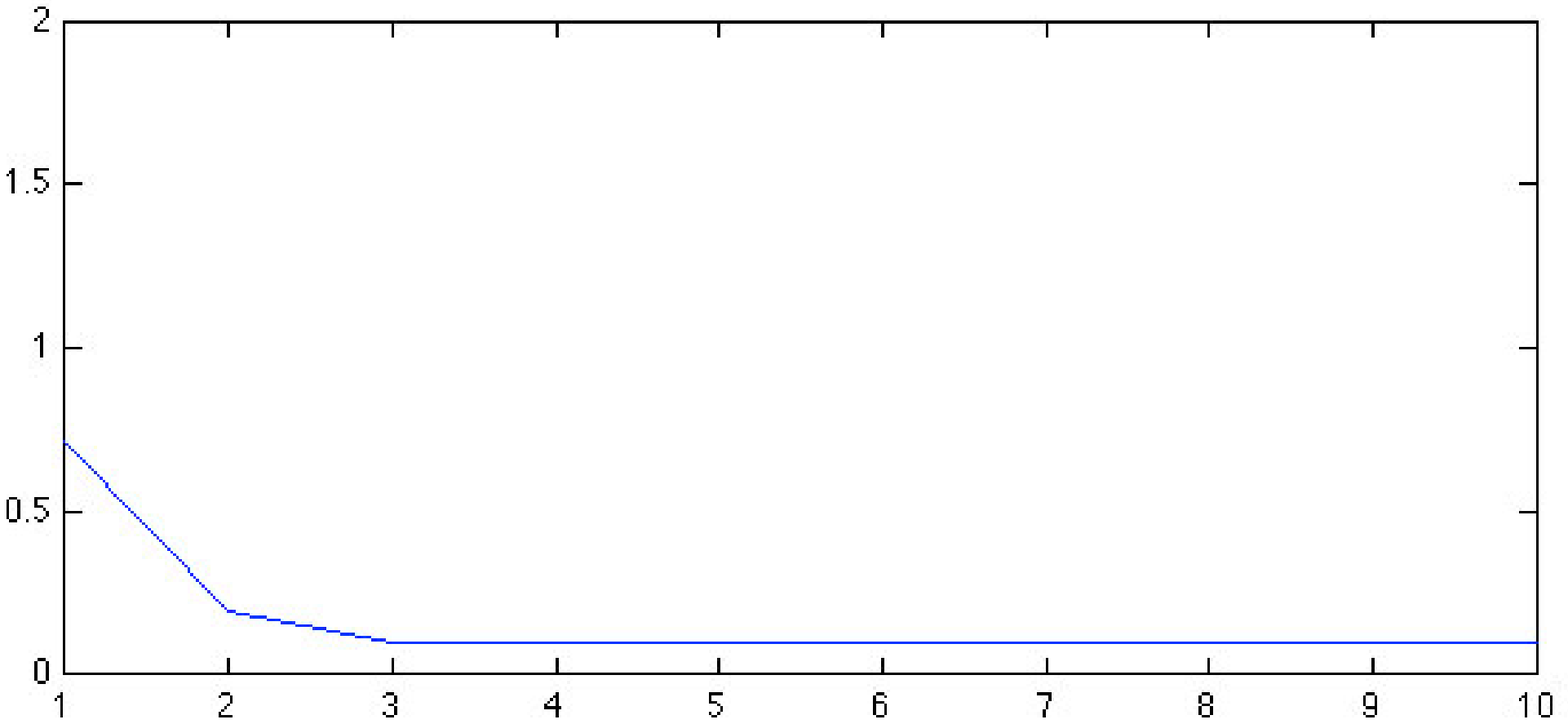}} \\
\subfigure[Initialization \#2 (random init)]{\includegraphics[width=4cm]{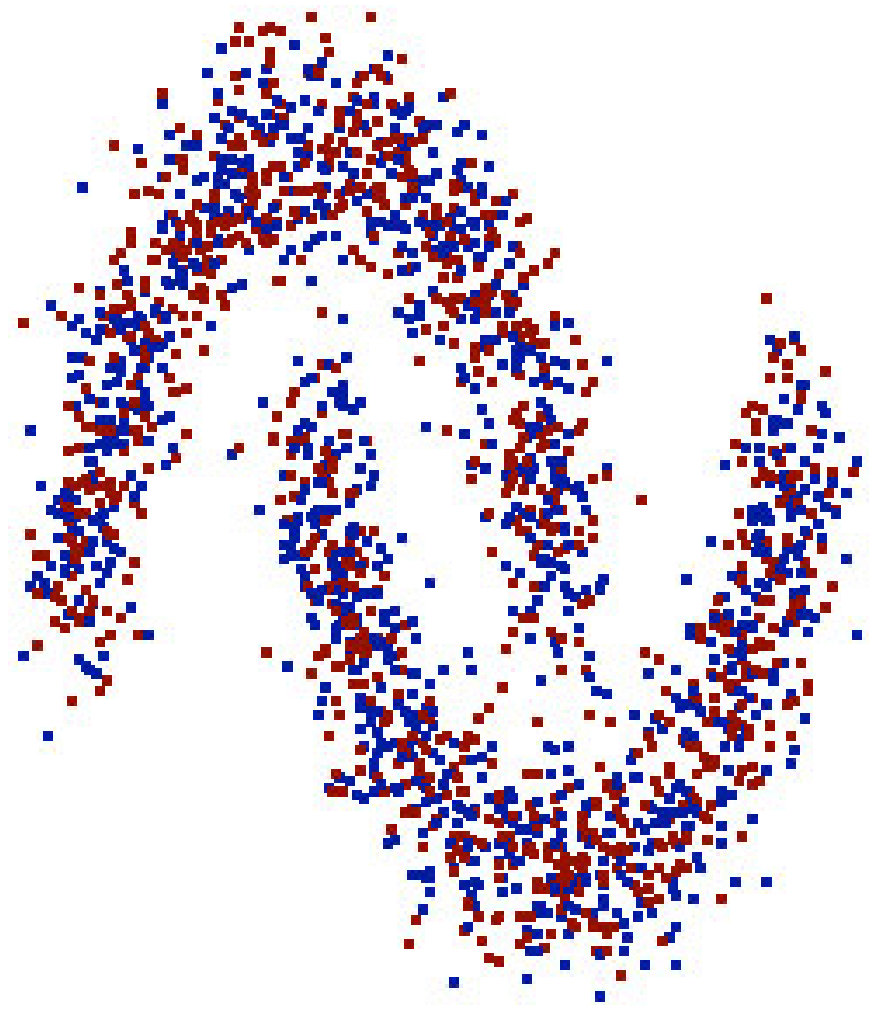}} 
\hspace{0.5cm}
\subfigure[Outcome of Algorithm \ref{alg-ratiocut}]{\includegraphics[width=4cm]{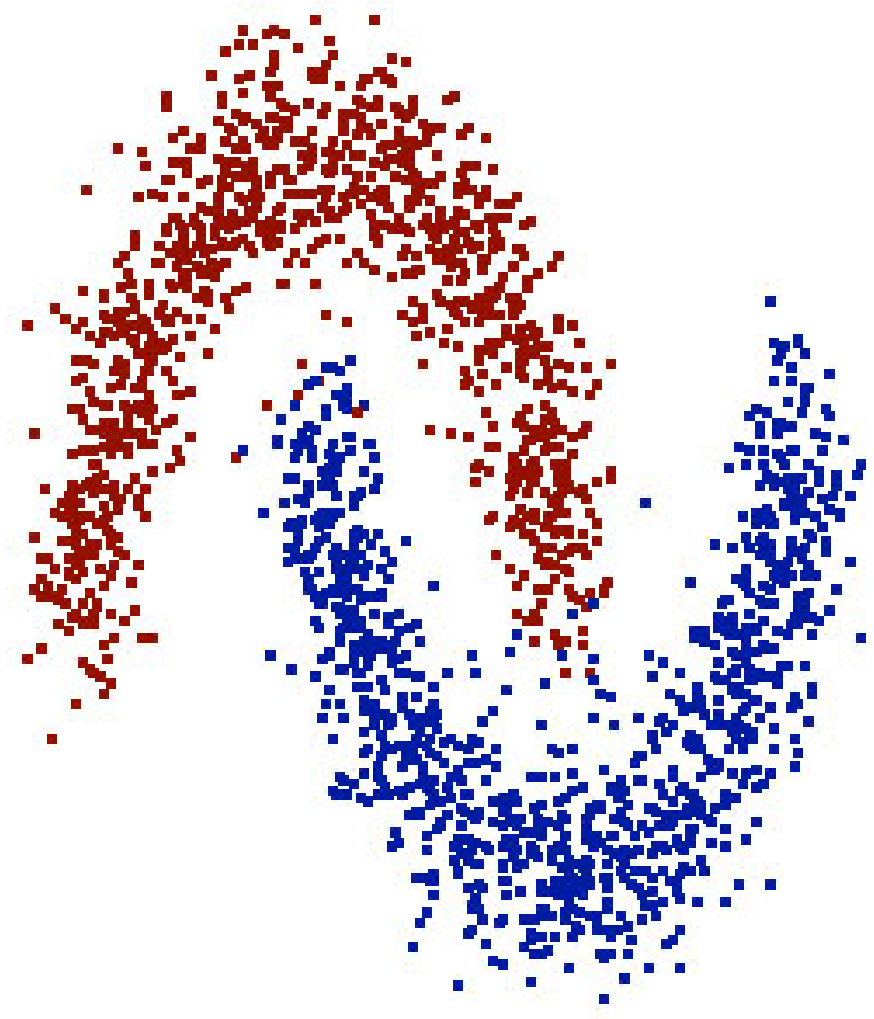}} 
\hspace{0.5cm}
\subfigure[Energy w.r.t. iteration]{\includegraphics[width=4cm]{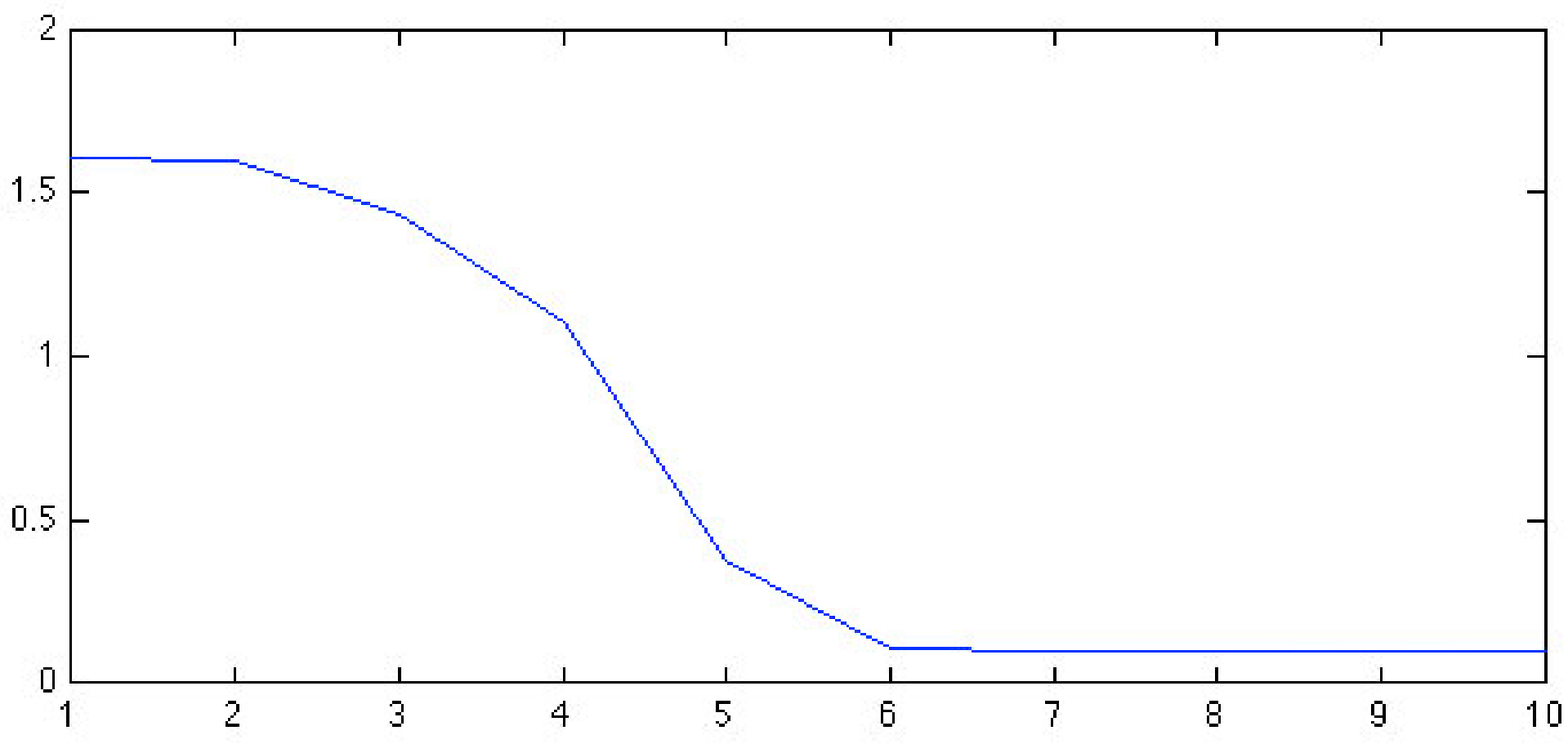}}\\
\subfigure[Initialization \#3 (random init)]{\includegraphics[width=4cm]{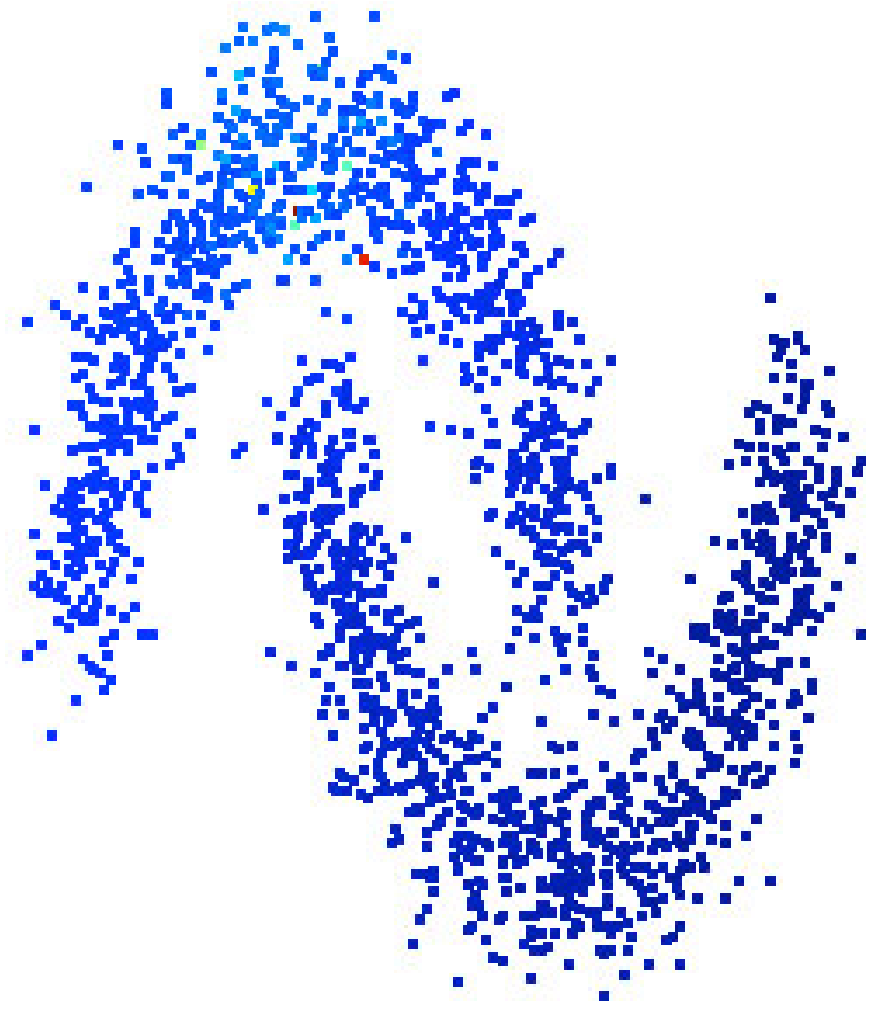}} 
\hspace{0.5cm}
\subfigure[Outcome of Algorithm \ref{alg-ratiocut}]{\includegraphics[width=4cm]{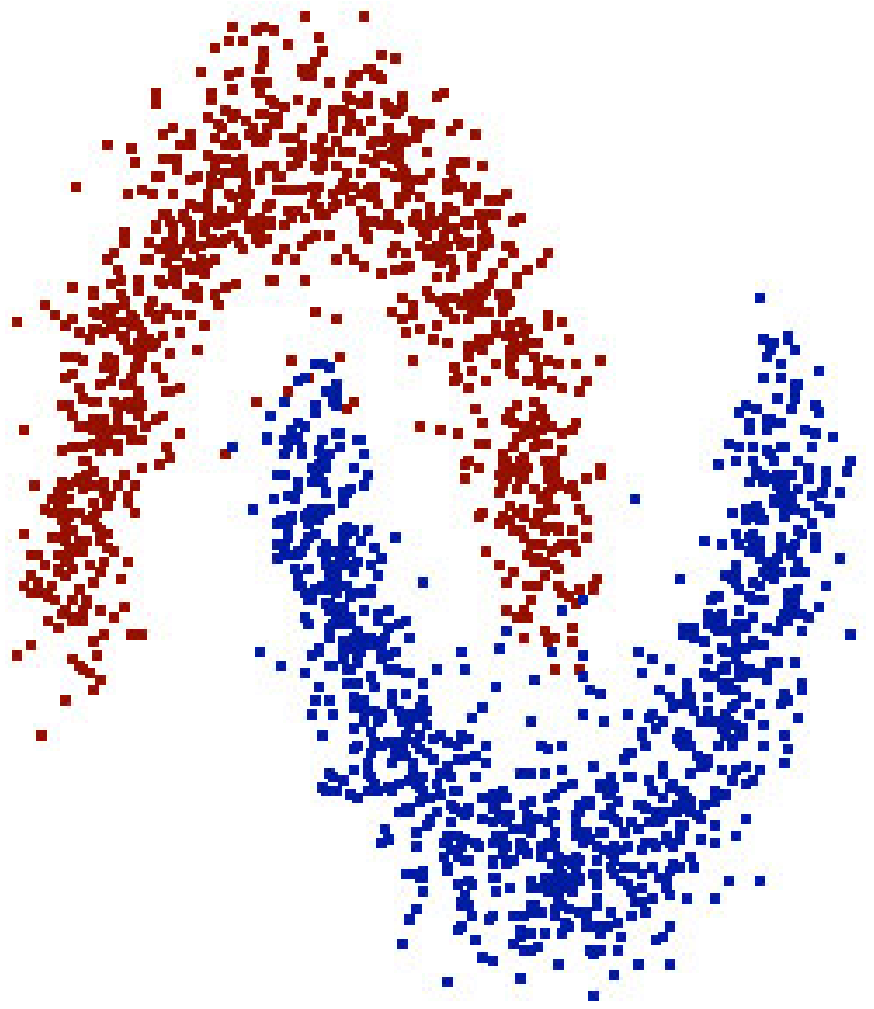}} 
\hspace{0.5cm}
\subfigure[Energy w.r.t. iteration]{\includegraphics[width=4cm]{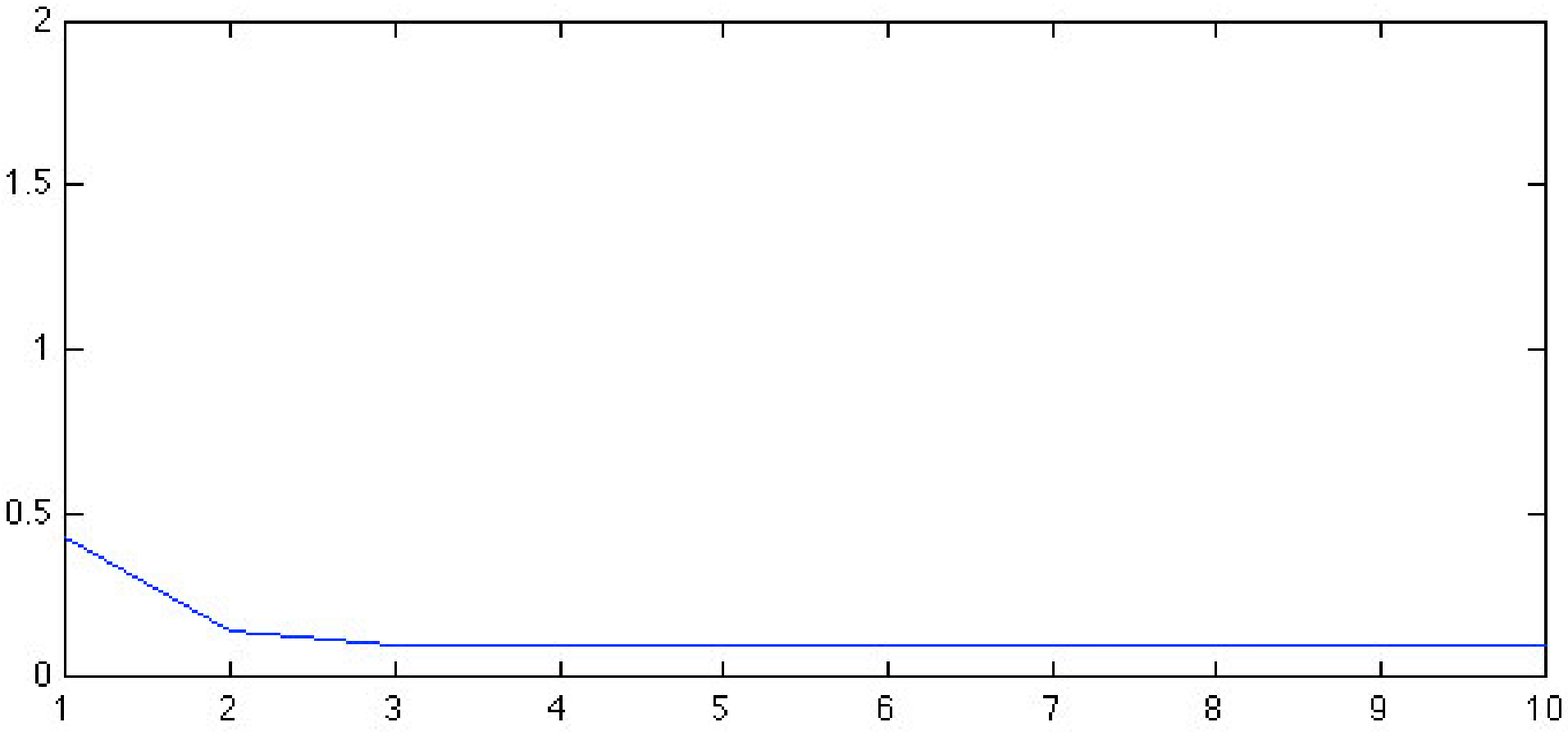}}  \\
\subfigure[Initialization \#4 (random init)]{\includegraphics[width=4cm]{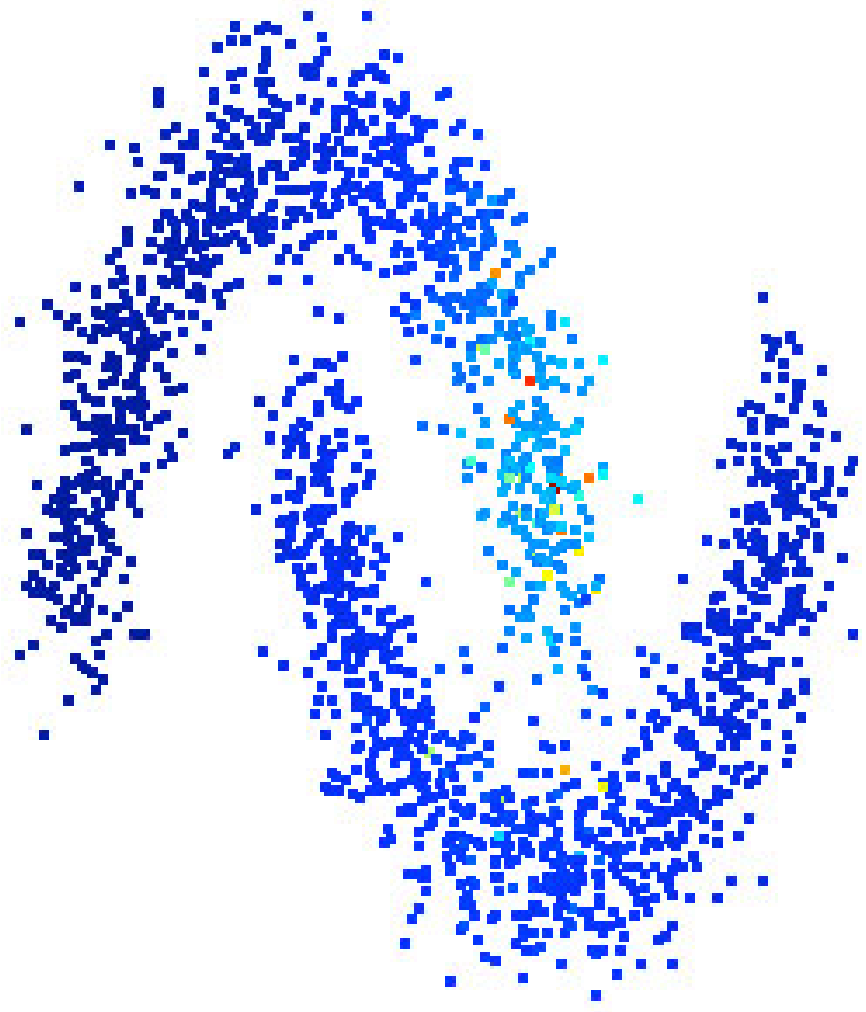}} 
\hspace{0.5cm}
\subfigure[Outcome of Algorithm \ref{alg-ratiocut}]{\includegraphics[width=4cm]{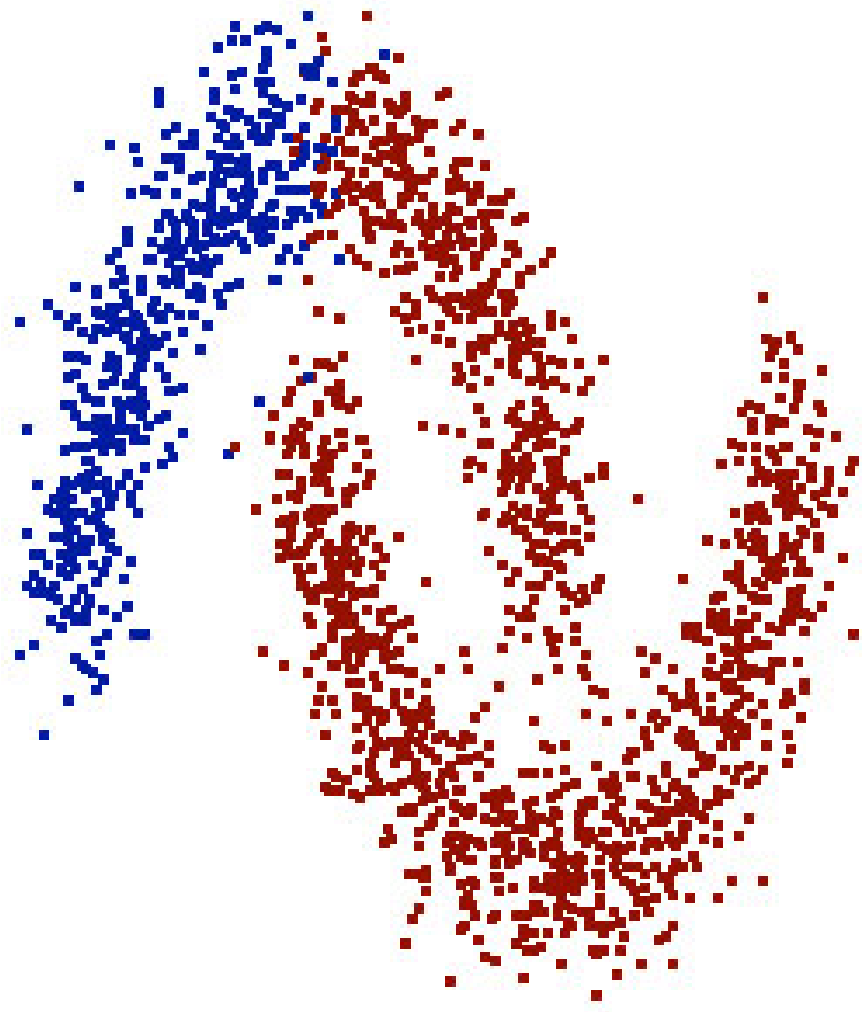}} 
\hspace{0.5cm}
\subfigure[Energy w.r.t. iteration]{\includegraphics[width=4cm]{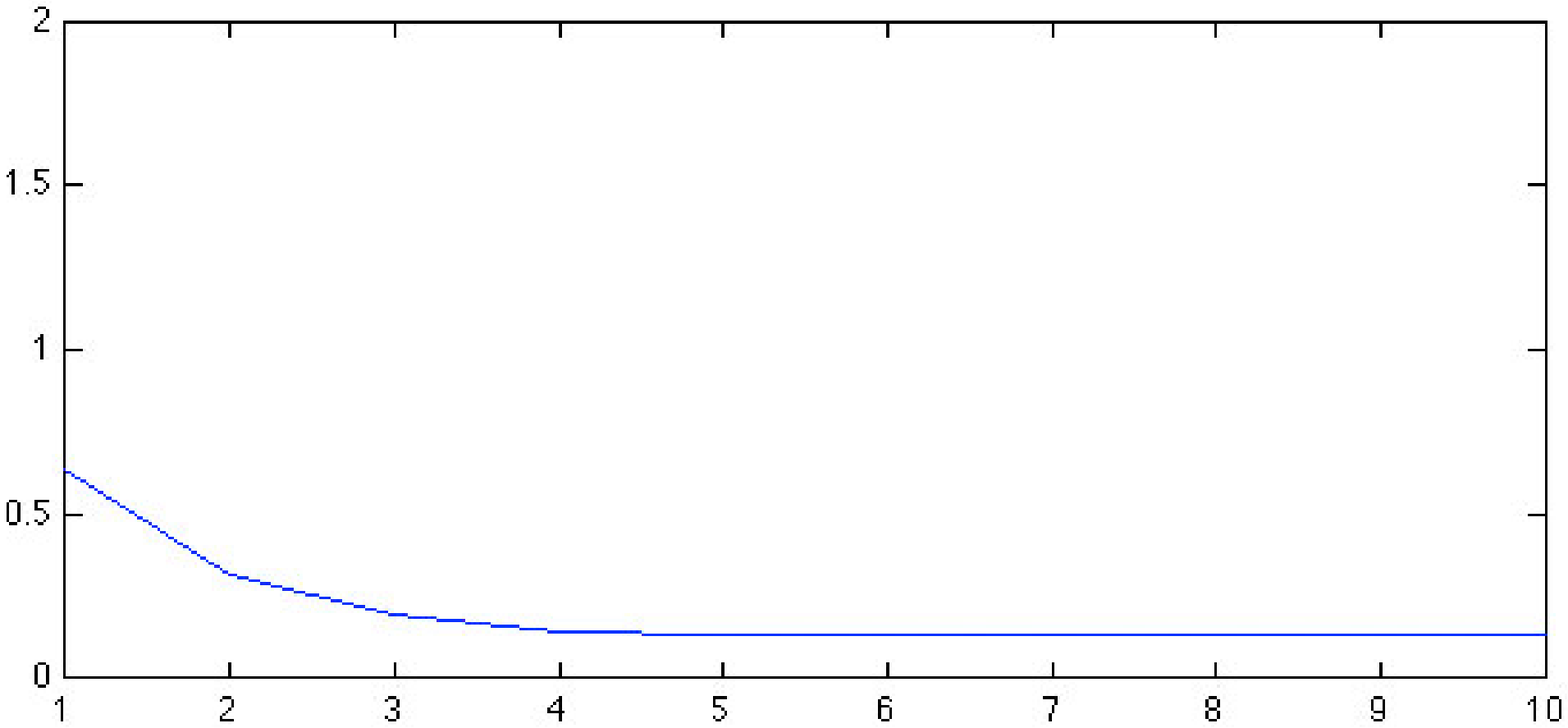}} 
\caption{Outcomes of Algorithm \ref{alg-ratiocut} with different initial data. On the right column the value of the ratio cut functional \eqref{rc2} is plotted versus the number of iterations.}
\label{experiments}
\end{figure}



\bibliographystyle{plain}

\bibliography{bib_nips}

\begin{thebibliography}{10}

\bibitem{art:BressonTaiChanSzlam12TransLearn}
X.~Bresson, X.-C. Tai, T.F. Chan, and A.~Szlam.
\newblock {Multi-Class Transductive Learning based on $\ell^1$ Relaxations of
  Cheeger Cut and Mumford-Shah-Potts Model}.
\newblock {\em UCLA CAM Report}, 2012.

\bibitem{pro:BuhlerHein09pLapla}
T.~B\"{u}hler and M.~Hein.
\newblock {Spectral Clustering Based on the Graph p-Laplacian}.
\newblock In {\em International Conference on Machine Learning}, pages 81--88,
  2009.

\bibitem{art:ChambollePock11FastPD}
A.~Chambolle and T.~Pock.
\newblock {A First-Order Primal-Dual Algorithm for Convex Problems with
  Applications to Imaging}.
\newblock {\em Journal of Mathematical Imaging and Vision}, 40(1):120--145,
  2011.

\bibitem{art:GoldsteinOsher09SB}
T.~Goldstein and S.~Osher.
\newblock {The Split Bregman Method for L1-Regularized Problems}.
\newblock {\em SIAM Journal on Imaging Sciences}, 2(2):323--343, 2009.

\bibitem{ratiocut}
L.~Hagen and A.~Kahng.
\newblock {New spectral methods for ratio cut partitioning and clustering.}
\newblock {\em IEEE Trans. Computer-Aided Design}, 11:1074 --1085, 1992.

\bibitem{pro:HeinBuhler10OneSpec}
M.~Hein and T.~B\"{u}hler.
\newblock {An Inverse Power Method for Nonlinear Eigenproblems with
  Applications in 1-Spectral Clustering and Sparse PCA}.
\newblock In {\em In Advances in Neural Information Processing Systems (NIPS)},
  pages 847--855, 2010.

\bibitem{pro:HeinSetzer11TightCheeger}
M.~Hein and S.~Setzer.
\newblock {Beyond Spectral Clustering - Tight Relaxations of Balanced Graph
  Cuts}.
\newblock In {\em In Advances in Neural Information Processing Systems (NIPS)},
  2011.

\bibitem{art:meyer76}
R.R. Meyer.
\newblock Sufficient conditions for the convergence of monotonic mathematical
  programming algorithms.
\newblock {\em Journal of Computer and System Sciences}, 12(1):108 -- 121,
  1976.

\bibitem{book:ostrowski}
A.~M. Ostrowski.
\newblock {\em Solution of Equations in Euclidean and Banach Spaces}.
\newblock Academic Press, New York, 1973.

\bibitem{pro:Rang-Hein-constrained}
S.~Rangapuram and M.~Hein.
\newblock {Constrained 1-Spectral Clustering}.
\newblock In {\em International conference on Artificial Intelligence and
  Statistics (AISTATS)}, pages 1143--1151, 2012.

\bibitem{art:RudinOsherFatemi92ROF}
L.~I. Rudin, S.~Osher, and E.~Fatemi.
\newblock {Nonlinear Total Variation Based Noise Removal Algorithms}.
\newblock {\em Physica D}, 60(1-4):259 -- 268, 1992.

\bibitem{pro:SzlamBresson10}
A.~Szlam and X.~Bresson.
\newblock Total variation and cheeger cuts.
\newblock In {\em Proceedings of the 27th International Conference on Machine
  Learning}, pages 1039--1046, 2010.

\bibitem{pro:ZelnikPerona04SelfTuning}
L.~Zelnik-Manor and P.~Perona.
\newblock {Self-tuning Spectral Clustering}.
\newblock In {\em In Advances in Neural Information Processing Systems (NIPS)},
  2004.

\end{thebibliography}

\end{document}